\documentclass[11pt, reqno]{amsart}
\usepackage{amsmath, amsthm, amscd, amsfonts, amssymb, graphicx}
\usepackage{ragged2e}
\usepackage[bookmarksnumbered, colorlinks, plainpages]{hyperref}
\usepackage{cite}
\newtheorem{theorem}{Theorem}[section]
\newtheorem{lemma}[theorem]{Lemma}

\newtheorem{Result}[theorem]{Result}
\theoremstyle{definition}

\theoremstyle{remark}
\newtheorem{remark}[theorem]{Remark}
\numberwithin{equation}{section}

\begin{document}
\setcounter{page}{1}

\centerline{}

\centerline{}

\title[]{Avoidance Criteria for Normal Holomorphic Curves on Complex Projective Space}

\author[Gopal Datt, Rahul Gogoi \MakeLowercase {and} Kushal Lalwani]{Gopal Datt, Rahul Gogoi \MakeLowercase {and} Kushal Lalwani}

\address{Department of Mathematics, Babasaheb Bhimrao Ambedkar University, Lucknow, India.}
\email{\textcolor[rgb]{0.00,0.00,0.84}{ggopal.datt@gmail.com, gopal.du@gmail.com}}
\address{Department of Mathematics, University of Delhi, Delhi, India.}
\email{\textcolor[rgb]{0.00,0.00,0.84}{rgogoi1729@gmail.com, rgogoi@maths.du.ac.in}}
\address{Department of Mathematics, Manav Rachna University, Faridabad, India}
\email{\textcolor[rgb]{0.00,0.00,0.84}{lalwani.kushal@gmail.com, kushallalwani@mru.edu.in}}


\begin{abstract}
We establish an avoidance criterion for families of holomorphic curves from the unit disk in complex plane to the complex
projective space that omit sufficiently many moving hypersurfaces in pointwise general position.  Furthermore, we 
study families of holomorphic curves that share hyperplanes and derive analogous normality conditions in this context.
\end{abstract}
\maketitle
\section{Introduction and Main Results}\label{Intro}
Recall that a family $\mathcal{F}$ of meromorphic functions on a
domain $D \subset \mathbb{C}$ is {\it normal} if given any sequence 
in $\mathcal{F}$ has a subsequence which converges compactly on $D$ 
to a meromorphic map into 
$\widehat{\mathbb{C}}=\mathbb C\,\cup\,\{\infty\}$. 
The changeover in focus from the families of meromorphic functions 
to individual functions occurred for the first time in the work of 
Yosida in 1934 \cite{yos-oacomf-1934}. This was subsequently improved upon by 
Noshiro \cite{nos-ctttomfituc-1938}. In 1957, Lehto and Virtanen established the 
modern definition of normal meromorphic functions, characterizing 
them through the properties of conformal invariance and the 
spherical derivative \cite{leh-bbanmf-1957}. Noshiro \cite{nos-ctttomfituc-1938} gave a 
characterization for a function to be normal, it states that a 
function $f$ meromorphic in $\mathbb{D}=\{z\in\mathbb C:|z|<1\}$ is 
{\bf normal} if and only if 
$$\sup_{z\in\mathbb{D}}(1-|z|^2)f^\#(z)<\infty,$$ 
where $\displaystyle f^\#(z)=\frac{|f'(z)|}{1+|f(z)|^2}$ denotes the spherical derivative of $f$. 
\smallskip

In \cite{ere-nhcfprtps-2007}, Eremenko generalized the concept of normality 
to holomorphic curves $f: \mathbb{D} \to \mathbb{P}^n(\mathbb{C})$, 
extending the theory beyond classical meromorphic functions to 
higher-dimensional complex projective spaces. 
\smallskip

Determining the conditions under which a meromorphic function or a 
family of such functions is {\bf normal} has been a central problem 
in complex analysis. Classically, this question is addressed 
through the criteria of value omission.
\smallskip

A notable contribution in this direction was made by Xu and Qiu \cite{qiu-aacfnf-2011}, who established a criterion based on 
the avoidance of three distinct functions defined on $\mathbb D$.

\begin{Result}[{\cite[Theorem 1]{qiu-aacfnf-2011}}] \label{R-Xu_Qiu_AC} 
Let $f$ be a meromorphic function in $\mathbb{D}$. Suppose that $\varphi_1$, $\varphi_2$ and $\varphi_3$ be three 
$\widehat{\mathbb C}$-valued meromorphic functions in $\mathbb D$ 
and are continuous on $\overline{\mathbb{D}}:=
\{z\in\mathbb C:|z|\leq1\}$ such that 
$\varphi_u(z) \neq \varphi_v(z)\ (1 \leq u < v \leq 3)$ 
on the unit circle $\partial\mathbb D=\{z\in\mathbb C:|z|=1\}$. 
If $f(z)\neq\varphi_u(z)$ for all $z\in\mathbb{D}$ and $u=1,2,3$, 
then f is normal.
\end{Result}

Yang \cite{yang-anotacfnf-2020} extended Result~\ref{R-Xu_Qiu_AC} by demonstrating that the 
local avoidance of three continuous boundary-distinct functions not 
only implies the normality of a single function but ensures the 
uniform normality of a family $\mathcal{F}$, as characterized by 
the bounded spherical derivative in the following result.

\begin{Result}[{\cite[Theorem 1.1]{yang-anotacfnf-2020}}] \label{R-Yang_AC} 
Let $\varphi_1$, $\varphi_2$ and $\varphi_3$ be three $\widehat{\mathbb C}$-valued meromorphic functions in 
$\mathbb D$ and are continuous on the $\overline{\mathbb D}$ where $\varphi_u(z)\neq\varphi_v(z)\ (1 \leq u < v \leq 3)$ on 
$\partial\mathbb D$. Let $\mathcal F$ be a family of meromorphic 
functions such that $f\neq\varphi_u$ on $\mathbb D$ for $u=1,2,3$ 
and $f\in\mathcal F$. 
Then there is a positive constant $M$ such that $$(1-|z|^2)f^\#(z)\leq M$$
for each $z\in\mathbb D$ and $f\in\mathcal F$.
\end{Result}
Subsequently Ahamed and Mandal in \cite{ahamed-cponmanhm-2023} gave 
a more general version of the above result. 
\begin{Result}[{\cite[Theorem 2.1]{ahamed-cponmanhm-2023}}] \label{R-Ahamed_Mandal_AC} 
Let $\varphi_1,\varphi_2,\varphi_3$ be three $\widehat{\mathbb C}$-valued meromorphic functions in $\mathbb D$ and continuous 
on $\overline{\mathbb D}$, such that $\varphi_u(z)\neq\varphi_v(z)\ (1 \leq u < v \leq 3)$ on $\partial\mathbb D$. 
Let $\mathcal F$ be a family of meromorphic functions where $f\neq\varphi_u$ on $\mathbb D$ for $u=1,2,3$ and $f\in\mathcal F$.
Then there is an $M>0$ such that $$\left(\frac{a-b\,|z|^2}{c+d\,|z|^2}\right)f^\#(z)\leq M$$ for each $z\in \mathbb D$ and $f\in\mathcal F$, where $a,b,c,d\in \mathbb R$ and $abc\neq0$.
\end{Result}

\begin{remark}
Result~\ref{R-Ahamed_Mandal_AC} requires additional constraints for mathematical rigor. Specifically, the condition $c + d|z|^2 \neq 0$ is required for the term $\left(\frac{a-b|z|^2}{c+d|z|^2}\right)f^\#(z)$ to remain well-defined and bounded. Furthermore, the identity $a - b|z|^2 = (\sqrt{a} + \sqrt{b}|z|)(\sqrt{a} - \sqrt{b}|z|)$ employed in the original proof is valid only under the assumption $ab \geq 0$. 
\end{remark} 

Keeping these requirements in mind, we now present an improved 
version of the result for the sake of completeness and mathematical 
rigor. Subsequently, we establish an analogue of 
Result~\ref{R-Ahamed_Mandal_AC} in the setting of projective 
spaces. Specifically, we consider the normality of a family of 
holomorphic curves from $\mathbb{D}$ into 
$\mathbb{P}^n(\mathbb{C})$ that omit a collection of moving 
hypersurfaces in pointwise general position.

\begin{theorem}\label{T-ACFMHSO}
Let $\{Q_i(z)\}_{i=1}^{2n+1}$ be moving hypersurfaces in pointwise general position in $\mathbb P^n(\mathbb C)$. Let 
$\mathcal F$ be a family of holomorphic curves from $\mathbb D$ to $\mathbb P^n(\mathbb C)$ such that $f(\mathbb D)\subset \mathbb P^n(\mathbb C)\setminus \displaystyle\bigcup_{i=1}^{2n+1} Q_i(z)$, 
for all $f\in\mathcal F$, $z\in\mathbb D$ and $i=1,2,\dots,2n+1$. Then there exists $M>0$ such that 
\begin{align*}\left(\frac{a-b\,|z|^2}{c+d\,|z|^2}\right)\|f'(z)\|_{FS}\leq M, \label{ACFMHSO} \tag{\ref*{T-ACFMHSO}} \end{align*}
for each $z\in \mathbb D$ and $f\in\mathcal F$, where $a,b,c,d\in\mathbb R$, $abc\neq 0$, $ab\geq 0$ and $c+d\,|z|^2\neq0$, for all $z\in\overline{\mathbb D}$.
\end{theorem}

Note that the object $\|f'(z)\|_{FS}$ in the statement of the above theorem is the Fubini\,--\,Study derivative of $f$ at $z$ whose precise definition will be discussed in Section~\ref{Prelim}.
\smallskip

In connection with the aforementioned result, the following natural question arises :
\begin{itemize}
\item Whether Theorem~\ref{T-ACFMHSO} can be extended to the case where the family of holomorphic curves shares a collection of hypersurfaces or hyperplanes, rather than merely omitting them.
\end{itemize}

We end this section with the following theorem, which answers the above question in the affirmative for the case where hyperplanes are shared by any two holomorphic curves of the family.

\begin{theorem}\label{T-ACFFHPS}
Let $\{H_i\}_{i=1}^{2n+1}$ be hyperplanes in general position in $\mathbb P^n(\mathbb C)$. Let $\mathcal F$ be a family of 
holomorphic curves from $\mathbb D$ to $\mathbb P^n(\mathbb C)$ such that $f^{-1}(H_i)=g^{-1}(H_i)$ on $\mathbb D$, for all 
$f,g\in\mathcal F$ and $i=1,2,\dots,2n+1$. Then there exists $M>0$ such that 
\begin{align*}\left(\frac{a-b\,|z|^2}{c+d\,|z|^2}\right)\|f'(z)\|_{FS}\leq M, \label{ACFFHPS} \tag{\ref*{T-ACFFHPS}} \end{align*}
for each $z\in \mathbb D$ and $f\in\mathcal F$, where $a,b,c,d\in\mathbb R$, $abc\neq 0$, $ab\geq 0$ and 
$c+d\,|z|^2\neq0$, for all $z\in\overline{\mathbb D}$.
\end{theorem}

\section{Preliminaries}\label{Prelim}

 This section is devoted to a detailed elaboration of the concepts and terminology introduced in Section~\ref{Intro}, as well as the introduction of certain fundamental notions required for the subsequent proofs. For a more comprehensive background and detailed exposition, the interested reader is referred to\cite{ere-nhcfprtps-2007, ere-bcoh-2008, pang-axostfnf-2015, quang-wnaqfohc-2018, li-goptwmh-2021, deth-nfommoscvfmhiacps-2015, yang-anotacfnf-2020, yang-hmitcpswmh-2021, fuji-ofommitcps-1974}.\smallskip

Let $\mathbb{P}^n(\mathbb{C})$ denote the complex $n$-projective space; fixing the homogeneous coordinates $[u_0:u_1:\cdots:u_n]$
in $\mathbb P^n(\mathbb C)$ and letting $f : \mathbb{C} \to \mathbb{P}^n(\mathbb{C})$ be a holomorphic curve represented by 
$f(z) = [f_0(z) : f_1(z) : \dots : f_n(z)]$, 
we now briefly outline the notions of the Fubini\,--\,Study distance and norm, along with the definition of the Fubini\,--\,Study derivative (see \cite{ere-nhcfprtps-2007, quang-wnaqfohc-2018, yang-anotacfnf-2020, yang-hmitcpswmh-2021, fuji-ofommitcps-1974} for more details).

\begin{enumerate}
    \item The \emph{Fubini\,--\,Study distance} between any two points in $\mathbb P^n(\mathbb C)$ is given as
\begin{equation*}
d_{FS}(p,q) = \frac{\sum_{k,l=0}^n |p_k q_l - p_l q_k|}{\sqrt{\sum_{k=0}^n |p_k|^2} \, \sqrt{\sum_{l=0}^n |q_l|^2}},
\end{equation*}
for $p = [p_0 : p_1 : \dots : p_n]$ and $q = [q_0 : q_1 : \dots : q_n]$ in $\mathbb{P}^n(\mathbb{C})$. \smallskip
    \item The \emph{Fubini\,--\,Study norm} of $f$ at $z$ is given as \begin{equation*}
\|f(z)\|_{FS} := \left( \sum_{k=0}^n |f_k(z)|^2 \right)^{1/2}.
\end{equation*}
    \item The \emph{Fubini\,--\,Study derivative} of $f$ at $z$ is given as \begin{equation*}
\|f'(z)\|_{FS} = 
\frac{\left( \sum_{0 \le l_1 < l_2 \le n} 
\left| f_{l_1}'(z) f_{l_2}(z) - f_{l_1}(z) f_{l_2}'(z) \right|^2 \right)^{1/2}}
{\sum_{k=0}^n |f_k(z)|^2}.
\end{equation*}
\end{enumerate}\smallskip

We next turn our focus to the notions of representation of a holomorphic curve, normal holomorphic curves, fixed 
hypersurfaces and moving hypersurfaces (see \cite{ere-nhcfprtps-2007, ere-bcoh-2008, deth-nfommoscvfmhiacps-2015, yang-anotacfnf-2020, yang-hmitcpswmh-2021} for more details).

A \emph{representation} of a holomorphic curve $f : \mathbb{C} \to \mathbb{P}^n(\mathbb{C})$ is a holomorphic map 
$\mathbf f:\mathbb C\to\mathbb C^{n+1}\setminus\{0\}$ such that $\mathbf{f}(z)=(f_0(z),f_1(z), \dots, f_n(z))$, where for every $z\in\mathbb C$, $f(z)=[f_0(z):f_1(z):\cdots:f_n(z)]$. That is, $\mathbf f$ is a holomorphic lift of $f$ with respect to the canonical projection $\pi:\mathbb C^{n+1}\setminus\{0\}\to\mathbb P^n(\mathbb C)$, $\pi(u_0,u_1,\dots,u_n)=[u_0:u_1:\cdots:u_n]$ such that $f=\pi\,\circ\,\mathbf f$.
\smallskip

A holomorphic curve $f:\mathbb{D} \to \mathbb P^n(\mathbb C)$ is \emph{normal} if
\begin{equation*}\sup_{z\in \mathbb{D}}(1-|z|^2)\|f'(z)\|_{FS}<\infty.\end{equation*}

A \textit{hypersurface} of degree $d\geq 1$ in $\mathbb{P}^n(\mathbb{C})$ is defined by
\begin{equation*}\left\{ [u_0:u_1:\cdots:u_n]\in\mathbb P^n(\mathbb C)\ |\ Q(u_0,u_1\ldots,u_n)= \sum_{I \in \tau_d} \alpha_I \, u^I = 0\right\},\end{equation*}
where each $\alpha_I$ are constants, $\tau_d = \left\{ (l_0,\ldots,l_n) \in \mathbb Z_{+}^{n+1} : \sum_{k=0}^n l_k= d \right\}$ 
and $u^I = \prod_{k=0}^n u_k^{l_k}$ for 
$I = (l_0,\ldots,l_n)$. \smallskip

The norm of $Q$ is given by
$$\|Q\|=\left( \sum_{I \in \tau_d} |\alpha_I|^2 \right)^{1/2}.$$\smallskip

Restricting to the case that $\|Q\|\neq 0$. Let $Q_1,Q_2,\ldots,Q_{\tilde q}$ $(\tilde q \geq t+1)$ be hypersurfaces 
in $\mathbb P^n(\mathbb{C})$ and $X\subseteq\mathbb P^n(\mathbb C)$ be a closed set.
The hypersurfaces $Q_1,Q_2,\ldots,Q_{\tilde q}$ are said to be in $t$-
subgeneral position with respect to $X$, if given any $1\le l_0\le\cdots\le l_t\le \tilde q$, \begin{equation*}
    X\,\cap\,Q_{l_0}\,\cap\cdots\,\cap\,Q_{l_t}=\varnothing. 
\end{equation*}
Also, the hypersurfaces $Q_1,Q_2,\dots,Q_{\tilde q}$ are said to be in general position if they are in $n$-subgeneral position with respect to $\mathbb P^n(\mathbb C)$.\smallskip

A \emph{moving hypersurface} of degree $d\geq 1$ in $\mathbb{P}^n(\mathbb{C})$
is defined by,
\begin{equation*}\left\{[u_0:u_1:\cdots:u_n]\in\mathbb P^n(\mathbb C)\ |\ Q(z;u_0,u_1,\ldots,u_n) \;=\; 
\sum_{I \in \tau_d} \alpha_I(z) \, u^I=0,\ z \in D\right\}\end{equation*}
where each coefficient $\alpha_I(z)$ is a holomorphic functions 
on the planar domain ${D}$ and not all $\alpha_I(z)$ vanish simultaneously for any 
$z\in D$. \smallskip
The norm of $Q$ at each point is given by
$$\|Q(z)\| = \Big( \sum_{I \in \tau_d} |\alpha_I(z)|^2 \Big)^{1/2},\ z\in D.$$\smallskip

Restricting to the case that $\|Q(z)\|\neq 0$. Let $Q_1(z),Q_2(z),\ldots,Q_{\hat q}(z)$ $(\hat q \geq n+1)$ be moving hypersurfaces in $\mathbb P^n(\mathbb{C})$, defined on a domain ${D}$. 
For each $z \in {D}$, we set
\begin{equation*}
\mathcal D(Q_1,Q_2,\ldots,Q_{\hat q})(z) \;=\; 
\prod_{1 \leq l_0 < \cdots < l_{n} \leq n}
\;\;\inf_{\substack{\mathbf u\in \mathbb{C}^{n+1}\\ \|\mathbf u\|=1\ \ }}
\left(\sum_{k=0}^{n}|Q_{l_k}(z;\mathbf u)|^2  \right).
\end{equation*}
If $\mathcal D(Q_1,Q_2,\ldots,Q_{\hat q})(z) > 0$ for some $z \in D$, 
we say that the hypersurfaces $Q_1(z),Q_2(z),$ $\ldots,Q_{\hat q}(z)$ are in (weakly)
general position or pointwise general position at $z$.
\begin{remark}
Hypersurfaces of degree $1$ are referred to as hyperplanes. In what follows,  hyperplanes will be denoted by $H$.
\end{remark}

Let $H \subset \mathbb{P}^n(\mathbb{C})$ be a hyperplane defined by
\begin{equation*}
H = \left\{ [u_0 : u_1 : \cdots : u_n] \in \mathbb{P}^n(\mathbb{C}) \; : \; \sum_{k=0}^n a_k u_k = 0 \right\},
\end{equation*}
where the coefficients $a_0, \dots, a_n \in \mathbb{C}$ are not all zero.
Set
$\langle \mathbf{f}, H \rangle := \sum_{k=0}^n a_k f_k.$
which is a holomorphic function on $D$ and its zero set coincides with $\{z\in D: f(z)\in H\}$.

\section{Essential Lemmas}
Before proceeding to the proofs of the main theorems, we present a collection of preliminary lemmas that will be utilized throughout the subsequent sections. \smallskip

The lemma below of Thai et al. is a generalization of Zalcman's rescaling lemma in higher dimensions, but formulated specifically for complex projective space. It seeks to aid \cite[Theorem 3.1]{ala-acfnicn-1991} of 
Aladro and Krantz, but they proved it under weaker assumptions and missed to consider when the target manifold can be non-compact.  
\begin{lemma}[{\cite[Theorem 2.5]{thai-fonmiscvahocs-2003}}] \label{L-Thai_Zalcman_Projective} 
Let $\mathcal{F}$ be a family of holomorphic mappings from a domain $\Omega\subset \mathbb{C}^m$ into $\mathbb{P}^n(\mathbb{C})$, where $m\in\mathbb N$. The family 
$\mathcal{F}$ is not normal on $\Omega$ if and only if there exist sequences $\{f_{\nu}\}_{\nu=1}^\infty \subset \mathcal{F}$, $\{z_{\nu}\}_{\nu=1}^\infty \subset \Omega$ with 
$z_{\nu} \to z_0 \in \Omega$ and $\{\rho_{\nu}\}$ with $\rho_{\nu}\searrow 0$ such that
$h_{\nu}(\zeta) := f_{\nu}(z_{\nu} + \rho_{\nu} \zeta)$
converges locally uniformly on $\mathbb{C}^m$ to a non-constant holomorphic mapping $h$ of $\mathbb{C}^m$ into 
$\mathbb{P}^n(\mathbb{C})$.
\end{lemma}
Eremenko obtained a Picard-type theorem \cite[Theorem 1]{ere-apttfhc-1999} for holomorphic curves into complex projective 
spaces where the curve avoids a sufficiently generic set of hypersurfaces; this forces the holomorphic curve to be a constant mapping. 
\begin{lemma}[{\cite[Theorem 1]{ere-apttfhc-1999}}] \label{L-Ere_omitting_hypersurface_general_position_implies_constant} 
Let $M \subseteq \mathbb{P}^n(\mathbb{C})$ be a closed subset  and say $\{Q_{j}\}_{j=1}^{2t+1}$ be a set of hypersurfaces in 
$\mathbb{P}^n(\mathbb{C})$ where
\begin{equation*}
M \cap \bigcap_{j \in I} Q_{j} = \varnothing
\quad \text{for every } I \subset \{1,2,\ldots,2t+1\},\ |I| = t+1.
\end{equation*}
Then any holomorphic mapping $f$ from $\mathbb{C}$ into
$M \setminus \bigcup_{j=1}^{2t+1} Q_{j}$ is constant.
\end{lemma}
Also Liu et al. introduced a lemma \cite[Lemma 3.7]{liu-ofommitcps-2016} to prove that the holomorphic map into the projective space would 
be rational if a definite number of hyperplanes which are in subgeneral position, were either missed only finitely many times
or had the image of the map lying on them. The authors subsequently gave a lemma \cite[Lemma 3.8]{liu-ofommitcps-2016} to 
prove that the holomorphic map is constant if a large number of the mentioned hyperplanes either omit the image or contain it.
\begin{lemma}[{\cite[Lemma 3.8]{liu-ofommitcps-2016}}] \label{L-Liu_hyperplanes_subgeneral_position_some_omit_some_intersect_finitely_implies_constant} 
Let $\{H_{j}\}_{j=1}^{2t+1}$ be hyperplanes in $\mathbb{P}^n(\mathbb{C})$ located in $t$\,-\,subgeneral 
position. Let $\hat{i}$ be an integer with $t+1 \leq \hat{i} \leq 2t+1$ and let $f : \mathbb{C} \to \mathbb{P}^n(\mathbb{C})$ be a 
holomorphic map. Assume that
\begin{enumerate}
\item given any $j=1,2,\dots,\hat{i}$, either $f(\mathbb{C}) \subset H_j$ or $f$ omits $H_j$ and
\item given any $j=\hat{i}+1,\dots,2t+1$, $\langle \mathbf f,H_j\rangle$ has atmost finitely many zeros in $\mathbb{C}$
\end{enumerate}
then $f$ is constant.
\end{lemma}
\section{Proof of the Main Results}

\begin{proof}[{\bf Proof of Theorem~\ref{T-ACFMHSO}}]
Suppose that the statement fails to hold for $z\in \mathbb D$ and $f\in\mathcal F$. Then there exist $\{f_{\nu}\}_{\nu=1}^\infty\subset \mathcal F$, $\{z_{\nu}\}_{\nu=1}^\infty\subset\mathbb D$ such that 
\begin{align*}
    \left(\frac{a-b\,|z_{\nu}|^2}{c+d\,|z_{\nu}|^2}\right)\|f_{\nu}'(z_{\nu})\|_{FS}\to\infty\ \text{ as }\ \nu\to\infty.\label{ACFMHSO.1} \tag{\ref*{T-ACFMHSO}.1}
\end{align*}
Thus \begin{align*}
    \left(\frac{\sqrt{a}-\sqrt{b}\,|z_{\nu}|}{c+d\,|z_{\nu}|^2}\right)\|f_{\nu}'(z_{\nu})\|_{FS}&=\left(\frac{a-b\,|z_{\nu}|^2}{(\sqrt{a}+\sqrt{b}\,|z_{\nu}|)(c+d\,|z_{\nu}|^2)}\right)\|f_{\nu}'(z_{\nu})\|_{FS}\\
    &\geq \frac{1}{\sqrt{a}+\sqrt{b}}\left(\frac{a-b|z_{\nu}|^2}{c+d|z_{\nu}|^2}\right)\|f_{\nu}'(z_{\nu})\|_{FS}.
\end{align*}
Hence \begin{align*}
\left(\frac{\sqrt{a}-\sqrt{b}\,|z_{\nu}|}{c+d\,|z_{\nu}|^2}\right)\|f_{\nu}'(z_{\nu})\|_{FS}\to\infty\ \text{ as }\ \nu\to\infty. \label{ACFMHSO.2} \tag{\ref*{T-ACFMHSO}.2} \end{align*}
Define $g_{\nu}(z):=f_{\nu}\left(z_{\nu}+\left(\frac{\sqrt{a}-\sqrt{b}\,|z_{\nu}|}{c+d\,|z_{\nu}|^2}\right)z\right)$, where $|z|<\dfrac{1-|z_{\nu}|}{\left(\frac{\sqrt{a}-\sqrt{b}\,|z_{\nu}|}{c+d\,|z_{\nu}|^2}\right)}$.
Then \begin{align*} \|g_{\nu}'(z)\|_{FS}&=\left\|f'_{\nu}\left(z_{\nu}+\left(\frac{\sqrt{a}-\sqrt{b}\,|z_{\nu}|}{c+d\,|z_{\nu}|^2}\right)z\right)\right\|_{FS} \left(\frac{\sqrt{a}-\sqrt{b}\,|z_{\nu}|}{c+d\,|z_{\nu}|^2}\right)
\end{align*}
    In particular for $z=0$, we have
    \begin{align*}
    \|g_{\nu}'(0)\|_{FS}&= \|f'_{\nu}(z_{\nu})\|_{FS} \left(\frac{\sqrt{a}-\sqrt{b}\,|z_{\nu}|}{c+d\,|z_{\nu}|^2}\right)\to\infty\ \text{ as }\ \nu\to\infty.
\end{align*}
Now, from Marty's theorem \cite[p. 4]{ere-nhcfprtps-2007}, it follows that $\{g_{\nu}\}$ is not normal at the origin.
By Lemma \ref{L-Thai_Zalcman_Projective}, there is a subsequence of $\{g_{\nu}\}$ with $w_{\nu}\to 0$ and $\rho_{\nu}\searrow 0$ such that
\begin{align*}
    G_{\nu}(\zeta):=&\,g_{\nu}(w_{\nu}+\rho_{\nu}\zeta)\\
    =&\,f_{\nu}\left(z_{\nu}+\left(\frac{\sqrt{a}-\sqrt{b}\,|z_{\nu}|}{c+d\,|z_{\nu}|^2}\right)w_{\nu}+\left(\frac{\sqrt{a}-\sqrt{b}\,|z_{\nu}|}{c+d\,|z_{\nu}|^2}\right)\rho_{\nu}\zeta\right)
\end{align*}
where $\zeta\in\mathbb C$ satisfies $w_{\nu}+\rho_{\nu}\zeta\in\mathbb D$, converges locally 
uniformly on $\mathbb C$ to a non-constant holomorphic curve $G(\zeta)$ from $\mathbb C$ to $\mathbb P^n(\mathbb C)$. 
Considering a subsequence, we assume that $z_{\nu}\to z_0$, where $|z_0|\leq 1$. We define,
$$\widetilde{z_{\nu}}=z_{\nu}+\left(\frac{\sqrt{a}-\sqrt{b}\,|z_{\nu}|}{c+d\,|z_{\nu}|^2}\right)w_{\nu}\ \text{ and }\ \widetilde{\rho_{\nu}}=\left(\frac{\sqrt{a}-\sqrt{b}\,|z_{\nu}|}{c+d\,|z_{\nu}|^2}\right)\rho_{\nu}.$$ 
It is clear that $\widetilde{z_{\nu}}\to z_0$ and $\widetilde\rho_{\nu}\to 0$ as $\nu\to\infty$.\smallskip

Consider the moving hypersurfaces $Q_i(z)$ for $i=1,2,\dots,2n+1$. Each $Q_i(z)$ is defined as,
 \begin{align*}Q_i(z)=&\left\{[u_0:u_1:\cdots:u_n]\in\mathbb P^n(\mathbb C):\sum_{I\in \tau_d} \alpha_I(z)\, u^I=0\right\}\\ =&~\{[u_0:u_1:\cdots:u_n]\in\mathbb P^n(\mathbb C):S_i(z,\mathbf{u})=0\},\end{align*}
where each coefficient $\alpha_I(z)$ is a holomorphic functions 
on $\mathbb{D}$ and not all $\alpha_I(z)$ vanish simultaneously for any $z\in \mathbb D$, $\tau_d = \left\{ (l_0,\ldots,l_n) \in \mathbb Z_{+}^{n+1} : \sum_{k=0}^n l_k= d \right\}$ 
and $u^I = \prod_{k=0}^nu_k^{l_k}$ for 
$I = (l_0,\ldots,l_n)$. Define the rescaled hypersurfaces $Q_{i,\nu}(\zeta):=Q_i(z_{\nu}+\rho_{\nu}\zeta)$. Since each $\alpha_I(z)$ are holomorphic and $z_{\nu}+\rho_{\nu}\zeta\to z_0$ locally uniformly on $\mathbb C$, 
we have $\alpha_I(z_{\nu}+\rho_{\nu}\zeta)\to \alpha_I(z_0)$ locally uniformly on $\mathbb C$. Hence $Q_{i,\nu}(\zeta)\to Q_i(z_0)$ locally uniformly. 
But each $f_{\nu}$ omits $Q_{i,\nu}(\zeta)$, so each $g_{\nu}$ omits $Q_{i,\nu}(\zeta)$ and thus each $G_{\nu}$ omits the hypersurface $Q_{i,\nu}(\zeta)$. \smallskip

If $z_0\in\mathbb D$, then for large $\nu$, $z_{\nu}\in D(z_0,r)$ for some $r>0$ with $D(z_0,r)\subset \mathbb D$. 
Now, $\widetilde{z_{\nu}}=z_{\nu}+\left(\frac{\sqrt{a}-\sqrt{b}\,|z_{\nu}|}{c+d\,|z_{\nu}|^2}\right)w_{\nu}$. As 
$w_{\nu}\to 0$ and $\left(\frac{\sqrt{a}-\sqrt{b}\,|z_{\nu}|}{c+d\,|z_{\nu}|^2}\right)$ stays bounded, we get 
$\widetilde{z_{\nu}}\to z_0\in\mathbb D$. Similarly $\widetilde \rho_{\nu}=\left(\frac{\sqrt{a}-\sqrt{b}\,|z_{\nu}|}{c+d\,|z_{\nu}|^2}\right)\rho_{\nu}\to 0$.
Thus for each compact set $K\subset\mathbb C$, for large $\nu$, $\widetilde z_{\nu} +\widetilde \rho_{\nu} K\subset D(z_0,r)\subset \mathbb D$. \smallskip

As each $f_{\nu}(\mathbb D)$ omits all hypersurface $Q_i(z_0)$, we have $S_i(z_0,\mathbf{f}_{\nu}(z))\neq 0$, for all 
$z\in\mathbb D$. Hence $S_i(z_0,\mathbf{G}_{\nu}(\zeta))\neq 0$, for all $\zeta\in\mathbb C$. Passing to a subsequence and applying Hurwitz's theorem,
\begin{itemize}
    \item $S_i(z_0,\mathbf{G})\equiv 0$ or,
    \item $S_i(z_0,\mathbf{G}(\zeta))\neq 0$ for all $\zeta\in\mathbb C$.
\end{itemize}
So, we have $G(\zeta)\in Q_i(z_0)$ for all $\zeta\in\mathbb C$ or $G(\zeta)\in\mathbb P^n(\mathbb C)\setminus Q_i(z_0)$
for all $\zeta\in\mathbb C$. %
\smallskip

Without loss of generality, assume there exists an integer $i_0$ such that $G(\zeta)\in\bigcap_{i=1}^{i_0} Q_i(z_0)$ for all 
$\zeta\in\mathbb C$ and $G(\zeta)\in\mathbb P^n(\mathbb C)\setminus \bigcup_{i=i_0+1}^{2n+1} Q_i(z_0)$, for all 
$\zeta\in\mathbb C$. Then, we get $G(\zeta)\in \bigcap_{i=1}^{i_0} Q_i(z_0)\setminus \bigcup_{i=i_0+1}^{2n+1} Q_i(z_0)$ 
for all $\zeta\in\mathbb C$. \smallskip

As we know, the hypersurfaces $\{Q_i(z)\}_{i=1}^{2n+1}$ are in pointwise general position in $\mathbb P^n(\mathbb C)$. Thus we
have $1\leq i_0\leq n$ and $M=\bigcap_{i=1}^{i_0} Q_i(z_0)$ is a closed subset of $\mathbb P^n(\mathbb C)$ with the property 
$$M\cap\left(\bigcap_{\hat{j}\in I} Q_{\hat{j}}(z_0)\right)=\varnothing,$$ for all $I\subset \{i_0+1,i_0+2,\dots, i_0+2(n-i_0)+1\}$ 
with $|I|=n+1-i_0$, (as $i_0+2(n-i_0)+1\leq 2n+1$).\smallskip

Thus the holomorphic curve $G:\mathbb C\to \mathbb P^n(\mathbb C)$ satisfies 
\begin{equation*}G(\zeta)\in M\setminus \bigcup_{\hat{j}=i_0+1}^{i_0+2(n-i_0)+1}Q_{\hat{j}}(z_0),\end{equation*}
for all $\zeta\in\mathbb C$. By Lemma \ref{L-Ere_omitting_hypersurface_general_position_implies_constant}, 
$G$ is a constant holomorphic curve from $\mathbb C$ to $\mathbb P^n(\mathbb C)$, which is a contradiction. \smallskip

If $z_0\in\partial\mathbb D$, as $z_{\nu}\to z_0$ so $|z_{\nu}|\to 1$. Recall 
$\widetilde{z_{\nu}}=z_{\nu}+\left(\frac{\sqrt{a}-\sqrt{b}\,|z_{\nu}|}{c+d\,|z_{\nu}|^2}\right)w_{\nu}$ and 
$\widetilde \rho_{\nu}=\left(\frac{\sqrt{a}-\sqrt{b}\,|z_{\nu}|}{c+d\,|z_{\nu}|^2}\right)\rho_{\nu}$. As $w_{\nu}\to 0$, we
still have $\widetilde{z_{\nu}}\to z_0\in\partial\mathbb D$. Thus for any compact set $K$ inside $\mathbb C$, 
$\widetilde{z_{\nu}}+\widetilde{\rho_{\nu}}K=z_{\nu}+\left(\frac{\sqrt{a}-\sqrt{b}\,|z_{\nu}|}{c+d\,|z_{\nu}|^2}\right)(w_{\nu}+\rho_{\nu} K)\subset\mathbb D$, 
for large $\nu$ and again we will be evaluating $G_{\nu}$, so the same procedure appears as above and $G$ turns out to 
constant, which is a contradiction. 
\end{proof}
\bigskip

\begin{proof}[{\bf Proof of Theorem~\ref{T-ACFFHPS}}]
Suppose that the statement fails to hold for $z\in \mathbb D$ and $f\in\mathcal F$. Then there exist $\{f_{\nu}\}_{\nu=1}^\infty\subset \mathcal F$, $\{z_{\nu}\}_{\nu=1}^\infty\subset\mathbb D$ such that 
\begin{align*}
    \left(\frac{a-b\,|z_{\nu}|^2}{c+d\,|z_{\nu}|^2}\right)\|f_{\nu}'(z)\|_{FS}\to\infty\ \text{ as }\ \nu\to\infty.\label{ACFFHPS.1} \tag{\ref*{T-ACFFHPS}.1}
\end{align*}
Thus \begin{align*}
    \left(\frac{\sqrt{a}-\sqrt{b}\,|z_{\nu}|}{c+d\,|z_{\nu}|^2}\right)\|f_{\nu}'(z_{\nu})\|_{FS}&=\left(\frac{a-b\,|z_{\nu}|^2}{(\sqrt{a}+\sqrt{b}\,|z_{\nu}|)(c+d\,|z_{\nu}|^2)}\right)\|f_{\nu}'(z_{\nu})\|_{FS}\\
    &\geq \frac{1}{\sqrt{a}+\sqrt{b}}\left(\frac{a-b|z_{\nu}|^2}{c+d|z_{\nu}|^2}\right)\|f_{\nu}'(z_{\nu})\|_{FS}.
\end{align*}
Hence \begin{align*}
\left(\frac{\sqrt{a}-\sqrt{b}\,|z_{\nu}|}{c+d\,|z_{\nu}|^2}\right)\|f_{\nu}'(z_{\nu})\|_{FS}\to\infty\ \text{ as }\ \nu\to\infty. \label{ACFFHPS.2} \tag{\ref*{T-ACFFHPS}.2} \end{align*}
Define $g_{\nu}(z):=f_{\nu}\left(z_{\nu}+\left(\frac{\sqrt{a}-\sqrt{b}\,|z_{\nu}|}{c+d\,|z_{\nu}|^2}\right)z\right)$, where $|z|<\dfrac{1-|z_{\nu}|}{\left(\frac{\sqrt{a}-\sqrt{b}\,|z_{\nu}|}{c+d\,|z_{\nu}|^2}\right)}$.
Then \begin{align*} \|g_{\nu}'(z)\|_{FS}&=\left(\frac{\sqrt{a}-\sqrt{b}\,|z_{\nu}|}{c+d\,|z_{\nu}|^2}\right)\left\|f'_{\nu}\left(z_{\nu}+\left(\frac{\sqrt{a}-\sqrt{b}\,|z_{\nu}|}{c+d\,|z_{\nu}|^2}\right)z\right)\right\|_{FS} 
\end{align*}
    In particular for $z=0$, we have
    \begin{align*}
    \|g_{\nu}'(0)\|_{FS}&= \left(\frac{\sqrt{a}-\sqrt{b}\,|z_{\nu}|}{c+d\,|z_{\nu}|^2}\right)\|f'_{\nu}(z_{\nu})\|_{FS} \to\infty\ \text{ as }\ \nu\to\infty.
\end{align*}
Now, from Marty's criterion \cite[p. 4]{ere-nhcfprtps-2007}, it follows that $\{g_{\nu}\}$ is not normal at origin. By Lemma 
\ref{L-Thai_Zalcman_Projective}, there is a subsequence of $\{g_{\nu}\}$ with $w_{\nu}\to 0$ and $\rho_{\nu}\searrow 0$
such that
\begin{align*}
    G_{\nu}(\zeta):=&\, g_{\nu}(w_{\nu}+\rho_{\nu}\zeta)\\
    =&\,f_{\nu}\left(z_{\nu}+\left(\frac{\sqrt{a}-\sqrt{b}\,|z_{\nu}|}{c+d\,|z_{\nu}|^2}\right)w_{\nu}+\left(\frac{\sqrt{a}-\sqrt{b}\,|z_{\nu}|}{c+d\,|z_{\nu}|^2}\right)\rho_{\nu}\zeta\right)
\end{align*}
where $\zeta\in\mathbb C$ satisfies $w_{\nu}+\rho_{\nu}\zeta\in\mathbb D$, converges locally 
uniformly on $\mathbb C$ to a non-constant holomorphic curve $G(\zeta)$ from $\mathbb C$ to $\mathbb P^n(\mathbb C)$. 
Considering a subsequence, we may now assume that $z_{\nu}\to z_0$, where $|z_0|\leq 1$. We define,
$$\widetilde{z_{\nu}}=z_{\nu}+\left(\frac{\sqrt{a}-\sqrt{b}\,|z_{\nu}|}{c+d\,|z_{\nu}|^2}\right)w_{\nu}\ \text{ and }\ \widetilde{\rho_{\nu}}=\left(\frac{\sqrt{a}-\sqrt{b}\,|z_{\nu}|}{c+d\,|z_{\nu}|^2}\right)\rho_{\nu}.$$ 
It is clear that $\widetilde{z_{\nu}}\to z_0$ and $\widetilde\rho_{\nu}\to 0$ as $\nu\to\infty$.\smallskip

Consider the hyperplanes $H_i$ for $i=1,2,\dots,2n+1$. Each $H_i$ is defined as,
\begin{align*}H_i=&\left\{[u_0:u_1:\cdots:u_n]\in\mathbb P^n(\mathbb C):\sum_{i=0}^n \alpha_I\,u^I=0\right\}\\ =&~\{[u_0:u_1:\cdots:u_n]\in\mathbb P^n(\mathbb C):L_i(\mathbf u)=0\}.\end{align*}

As given any $f,g\in \mathcal F$, $f$ and $g$ shares $H_i$, for all $i=1,2,\dots,2n+1$. Let 
$\{f_{\nu}\}_{\nu=1}^{\infty}\subset\mathcal F$. Then there exists $E_i\subset \mathbb D$, 
for all $i=1,2,\dots,2n+1$, $f_{\nu}^{-1}(H_i)=E_i$. \smallskip

If $E_i=\varnothing$, then $f_{\nu}^{-1}(H_i)=\varnothing$, for each $f_{\nu}\in\mathcal F$ and for all $i=1,2,\dots,2n+1$. If 
$z_0\in\mathbb D$, then for large $\nu$, $z_{\nu}\in D(z_0,r)$ for some $r>0$ with $D(z_0,r)\subset\mathbb D$. Now, 
$\widetilde{z_{\nu}}=z_{\nu}+\left(\frac{\sqrt{a}-\sqrt{b}\,|z_{\nu}|}{c+d\,|z_{\nu}|^2}\right)w_{\nu}$. As 
$w_{\nu}\to 0$ and $\left(\frac{\sqrt{a}-\sqrt{b}\,|z_{\nu}|}{c+d\,|z_{\nu}|^2}\right)$ stays bounded, we get 
$\widetilde{z_{\nu}}\to z_0\in\mathbb D$. Similarly $\widetilde \rho_{\nu}=\left(\frac{\sqrt{a}-\sqrt{b}\,|z_{\nu}|}{c+d\,|z_{\nu}|^2}\right)\rho_{\nu}\to 0$. 
Thus for each compact set $K\subset\mathbb C$, for large $\nu$, $\widetilde z_{\nu} +\widetilde \rho_{\nu} K\subset D(z_0,r)\subset \mathbb D$. \smallskip

As each $f_{\nu}(\mathbb D)$ omits all hyperplanes $H_i$, so $L_i(\mathbf{f}_{\nu}(z))\neq 0$, for all $z\in\mathbb D$. Hence
$L_i(\mathbf{G}_{\nu}(\zeta))\neq 0$, for all $\zeta\in\mathbb C$. Passing it to the limit and applying Hurwitz theorem,
\begin{itemize}
    \item $L_i(\mathbf{G}(\zeta))\equiv 0$ or,
    \item $L_i(\mathbf{G}(\zeta))\neq 0$ for all $\zeta\in\mathbb C$.
\end{itemize}

If $L_i(\mathbf{G}(\zeta))\equiv 0$ holds then there exists $i$ such that $G(\mathbb C)\subset H_i$ as 
the hyperplanes are in general position, their intersections with $H_i$ 
remains hyperplane in $H_i\equiv\mathbb P^{n-1}(\mathbb C)$. As there are $2n+1$ hyperplanes, there are atleast $2n$ hyperplanes in general position. Now observe 
$2n=2(n-1)+2$. So inductively, inside $H_i$, we still have atleast $2(n-1)+1$ hyperplanes. Thus $G$ must be constant. But by Lemma \ref{L-Thai_Zalcman_Projective}, 
$\|G'(0)\|_{FS}=1$, which is a contradiction. Hence $G(\mathbb C)\cap H_i=\varnothing$, for all $i=1,2,\dots,2n+1$.
Thus $G(\mathbb C)\subset \mathbb P^n(\mathbb C)\setminus \bigcup_{i=1}^{2n+1}H_i$. Also by Lemma \ref{L-Ere_omitting_hypersurface_general_position_implies_constant}, we get $G$ is constant, which is a contradiction. \smallskip

If $z_0\in\partial\mathbb D$, as $z_{\nu}\to z_0$, so $|z_{\nu}|\to 1$. Recall 
$\widetilde{z_{\nu}}=z_{\nu}+\left(\frac{\sqrt{a}-\sqrt{b}\,|z_{\nu}|}{c+d\,|z_{\nu}|^2}\right)w_{\nu}$ and 
$\widetilde \rho_{\nu}=\left(\frac{\sqrt{a}-\sqrt{b}\,|z_{\nu}|}{c+d\,|z_{\nu}|^2}\right)\rho_{\nu}$. As $w_{\nu}\to 0$, we 
still have $\widetilde{z_{\nu}}\to z_0\in\partial\mathbb D$. Thus for any compact set $K$ inside $\mathbb C$, 
$\widetilde{z_{\nu}}+\widetilde{\rho_{\nu}}K=z_{\nu}+\left(\frac{\sqrt{a}-\sqrt{b}\,|z_{\nu}|}{c+d\,|z_{\nu}|^2}\right)(w_{\nu}+\rho_{\nu} K)\subset\mathbb D$, 
the same procedure appears as above and $G$ turns out to 
constant by Lemma \ref{L-Ere_omitting_hypersurface_general_position_implies_constant}, which is again a contradiction. \smallskip

If $E_i\neq\varnothing$, then set $E_i=\{z\in\mathbb D:\langle\mathbf{f}_{\nu}(z),H_i\rangle=0\}$. Also here $E_i$'s
are discrete in $\mathbb D$, 
so given any compact set $K\subset \mathbb C$, the functions 
$\langle\mathbf{G}_{\nu}(\zeta),H_i\rangle$ have only finitely many zeros in $K$, for large $\nu$. \smallskip

Since $G_{\nu}\to G$ locally uniformly, we have $\langle\mathbf{G}_{\nu}(\zeta),H_i\rangle\to \langle\mathbf{G}(\zeta),H_i\rangle$ locally uniformly
on $\mathbb C$. Similarly, if $z_0\in\mathbb D$, then for each compact set $K\subset\mathbb C$, for large $\nu$,
$\widetilde z_{\nu} +\widetilde \rho_{\nu} K\subset D(z_0,r)\subset \mathbb D$.
Thus passing through a subsequence and applying Hurwitz's theorem, 
\begin{itemize}
    \item $G(\mathbb C)\subset H_i$ or,
    \item $\langle\mathbf{G}(\zeta),H_i\rangle$ has atmost finitely many zeros in $\mathbb C$.
\end{itemize}

Let $I=\{i:G(\mathbb C)\subset H_i\}$, $\widetilde{i}=|I|$. Since the hyperplanes are in general position in $\mathbb P^n(\mathbb C)$.
Thus atmost $n$ of them can have a common point. So $\widetilde{i}\leq n$. Define $J$ as the complement of the set $I$. Then $|J|=2n+1-\widetilde{i}\geq n+1$, for all 
$i\in J$, $\langle \mathbf{G}(\zeta),H_i\rangle$ has finitely many zeros. By Lemma 
\ref{L-Liu_hyperplanes_subgeneral_position_some_omit_some_intersect_finitely_implies_constant}, $G$ turns out to be a constant which is a contradiction. \smallskip

If $z_0\in\partial\mathbb D$, then again for any compact set $K$ inside $\mathbb C$, $$\widetilde{z_{\nu}}+\widetilde{\rho_{\nu}}K=z_{\nu}+\left(\frac{\sqrt{a}-\sqrt{b}\,|z_{\nu}|}{c+d\,|z_{\nu}|^2}\right)(w_{\nu}+\rho_{\nu} K)\subset\mathbb D$$ for large $\nu$ as above $G$ turns out to constant by Lemma
\ref{L-Liu_hyperplanes_subgeneral_position_some_omit_some_intersect_finitely_implies_constant}, 
which is again a contradiction. 
\end{proof}
\section*{Acknowledgement}

The authors would like to thank Prof. Sanjay Kumar Pant for valuable discussions and suggestions.
\bibliographystyle{amsplain}

\end{document}